\documentclass[11pt]{article}

\usepackage{epsfig}
\usepackage{mathrsfs}
\usepackage{color}
\usepackage[british,english]{babel}
\usepackage{amsthm}
\usepackage{amsmath}
\usepackage{amsfonts}
\usepackage{amssymb}

\usepackage{graphicx}
\newtheorem{corollary}{Corollary}[section]

\newtheorem{lemma}[corollary]{Lemma}
\newtheorem{proposition}[corollary]{Proposition}
\newtheorem{remark}[corollary]{Remark}
\newtheorem{theorem}[corollary]{Theorem}

\newfont{\sBlackboard}{msbm10 scaled 900}

\newcommand{\dd}     {{\rm d}}
\newcommand{\mylabel}[1]{\label{#1}
            \ifx\undefined\stillediting
            \else \fbox{$#1$}\fi }
\newcommand{\BE}{\begin{equation}}
\newcommand{\BEQ}[1]{\BE\mylabel{#1}}
\newcommand{\EEQ}{\end{equation}}
\newcommand{\rfb}[1]{\mbox{\rm
   (\ref{#1})}\ifx\undefined\stillediting\else:\fbox{$#1$}\fi}
\newcommand{\half}   {{\frac{1}{2}}}

\newfont{\Blackboard}{msbm10 scaled 1200}
\newcommand{\bl}[1]{\mbox{\Blackboard #1}}
\newfont{\roma}{cmr10 scaled 1200}

\def\CC{\rm \hbox{C\kern-.56em\raise.4ex
         \hbox{$\scriptscriptstyle |$}\kern+0.5 em }}

\newcommand{\rline}  {{\mathbb R}}

\newcommand{\zline}  {{\bl Z}}
\def\b{\beta}

\def\cH{{\cal H}}

\newcommand{\Dscr} {{\cal D}}

\newcommand{\mm}    {{\hbox{\hskip 0.5pt}}}
\newcommand{\m}     {{\hbox{\hskip 1pt}}}

\newcommand{\bluff} {{\hbox{\raise 15pt \hbox{\mm}}}}

\newcommand{\FORALL} {{\hbox{$\hskip 11mm \forall \;$}}}

\newcommand{\rarrow} {{\,\rightarrow\,}}

\makeatletter
\def\section{\@startsection {section}{1}{\z@}{-3.5ex plus -1ex minus
    -.2ex}{2.3ex plus .2ex}{\large\bf}}
\makeatother
\def\be{\begin{equation}}
\def\ee{\end{equation}}

\def\ds{\displaystyle}
\begin{document}
\thispagestyle{empty}
\title{\bf Stabilization of abstract thermo-elastic semigroup \footnote{Research supported by the IRD and Laboratoire de
Math\'ematiques et Dynamique de Populations, Universit\'e Cadi
Ayyad, Marrakech and by CMPTM-Research project : 10/TM /37.}}
\author{E. M. Ait Ben Hassi $^\ddag$, \, K. Ammari \thanks{UR Analyse et Contr\^ole des Edp (05/UR/15-01), D\'epartement de Math\'ematiques,
Facult\'e des Sciences de Monastir, Universit\'e de Monastir, 5019 Monastir, Tunisie,
e-mail~: kais.ammari@fsm.rnu.tn} ,  \, S. Boulite $^\ddag$ \, and \, L. Maniar
\thanks{
Universit\'e Cadi
Ayyad,  Facult\'e des Sciences Semlalia, LMDP, UMMISCO (IRD- UPMC), 
Marrakech 40000, B.P. 2390, Maroc, e-mail:
m.benhassi@uca.ma,  sboulite@uca.ma, 
maniar@uca.ma}}
\date{}
\maketitle
{\bf Abstract.} {\small
In this paper we characterize the stabilization for some thermo-elastic type system with Cattaneo law and we prove that the exponential or polynomial stability of this system implies a polynomial stability of the correspond thermoelastic system with the Fourier law. The proof of the main results uses, respectively, the methodology introduced in 
Ammari-Tucsnak \cite{ammari}, where the exponential stability for the closed loop problem is reduced to an observability
estimate for the corresponding uncontrolled system, and a characterization of the polynomial stability for a $C_0$-semigroup, in a Hilbert space,  by a polynomial estimation of the  resolvante of its generator obtained by Borichev-Tomilov \cite{tomilov}. An illustrating examples are given.}

\noindent
{\bf AMS subject classification (2010)}: 35B40, 47D06, 93B07, 93C25, 93D15.\\
{\bf Keywords}: abstract thermo-elastic system, Cattaneo law, Fourier law, exponential stability, observability inequality, polynomial stability.
\section{Introduction and main results} \label{intro}
Let $H_i$ be a Hilbert space equipped with the norm $\|\cdot\|_{H_i}, i=1,2$, and let
$A_1 :\Dscr(A_1) \subset H_1 \rightarrow H_1$ and $A :\Dscr(A) \subset H_2 \rightarrow H_2$ are positive self-adjoint   operators.

 We introduce the scale of Hilbert
spaces $H_{1,\alpha}$, $\alpha\in\rline$, as follows\m: for every
$\alpha\geq 0$, $H_{1,\alpha}=\Dscr(A_1^{\alpha})$, with the norm
$\|z \|_{1,\alpha}=\|A_1^\alpha z\|_{H_1}$ and $H_{2,\alpha}=\Dscr(A^{\alpha}),$ with the norm
$\|z \|_{2,\alpha}=\|A^{\alpha} z\|_{H_2}$. The space $H_{i,-\alpha}$ is
defined by duality with respect to the pivot space $H_i$ as
follows\m: $H_{i,-\alpha} =H_{i,\alpha}^*,$ for $\alpha>0, \,i=1,2$. The
operators $A_1$ and $A$ can be extended (or restricted) to each $H_{i,\alpha}$,
such that it becomes a bounded operator
\BEQ{A0ext}
A_1 : H_{1,\alpha} \rarrow H_{1,\alpha - 1},    \\
A : H_{2,\alpha} \rarrow H_{2,\alpha - 1},  \FORALL \alpha\in\rline.
\EEQ
We assume that the operator $A$ can be written as $A= A_2A_2^*$, 
where $A_2 \in {\cal L}(H_{1},H_{2, - \half})$,  which can be extended (or restricted) to $H_{1,\alpha}$,
such that it becomes an operator of ${\cal L}(H_{1,\alpha},H_{2,\alpha - \half}), \, \alpha \in \rline$,  and 
$A_2^* \in {\cal L}(H_{2},H_{1, -\half})$,  which can be extended (or restricted) to $H_{2,\alpha}$,
such that it becomes an operator of ${\cal L}(H_{2,\alpha},H_{1,\alpha - \half}), \, \alpha \in \rline.$
Let $C \in {\cal L}(H_{2},H_{1, - \half})$ and $C^* \in {\cal L}(H_{1,\half},H_{2})$,  which can be extended or restricted to $H_{2,\alpha}, H_{1,\alpha}$, such that it belongs to  ${\cal L}(H_{2,\alpha},H_{1,\alpha - \half}), {\cal L}(H_{1,\alpha},H_{2,\alpha - \half}), \, \alpha \in \rline,$ respectively. We denote by $H_1^\tau$ the space  $H_1$ equipped with the inner product $<u,v>_{H_1^\tau } = \tau \, <u,v>_{H_1}~, \, \, u,v \in H_1.$

We consider the following abstract thermo-elastic system with Cattaneo law 
\BEQ{damped1}
\ddot w_1(t) + A_1 w_1(t)  + C w_2(t)=\m 0, \,
\EEQ
\BEQ{damped2}
\dot w_2(t) + A_2 w_3(t) - C^* \dot w_1(t)=\m 0,
\EEQ
\BEQ{damped3}
\tau \, \dot w_3(t) + w_3 - A_2^* w_2(t) = \m 0,
\EEQ
\BEQ{output}
w_1(0) \m=\m w_1^0,  \m \dot w_1(0)=w_1^1,  \,
w_2(0) \m=\m w_2^0, \, w_3 (0) = w_3^0,
\EEQ
where $\tau >0$ is a constant and $t \in [0,\infty)$ is the time. The equations \rfb{damped1}- \rfb{damped3} are understood as equations in $H_{1,-\half}, H_{2,-\half}$ and $H_{1, - \half}$, respectively, i.e., all the terms
are in $H_{1,-\half}$, $H_{2,-\half}$ and $H_{1, - \half}$, respectively. We show the well-posedness of the 
 abstract  system \eqref{damped1}-\eqref{output} in the space ${\cal H} = H_{1,\half} \times H_1 \times H_2 \times H_1^\tau$.   Moreover, one can see that for regular solutions, the  energy of this system defined by  
$$
E(t) = \frac{1}{2} \, \left\|(w_1,\dot w_1,w_2,w_3) \right\|^2_{{\cal H_{\tau}}}, \,  \, t \geq 0,
$$
satisfies the following equality
\be
E(0) - E(t) = 
 \int_0^t
\left\|w_3(s)\right\|_{H_1}^2\dd s, \,  \, t \geq 0.
\label{ESTEN}
\ee
The aim of this paper is to show first that the exponential and polynomial decay of the energy $E(t)$ is reduced to an observability inequality for a corresponding conservative adjoint system, as in  \cite{Benhassi1, Benhassi2, ammari, haraux}. 

For $\tau=0$, the thermo-elastic problem with Cattaneo law \eqref{damped1}-\eqref{output} is just the following  classical thermo-elastic system (with Fourier law)   

\BEQ{damped11}
\ddot w_1(t) + A_1 w_1(t)  + C w_2(t)=\m 0, \,
\EEQ
\BEQ{damped21}
\dot w_2(t) + A w_2(t) - C^* \dot w_1(t)=\m 0,
\EEQ
\BEQ{output11}
w_1(0) \m=\m w_1^0,  \m \dot w_1(0)=w_1^1,  \,
w_2(0) \m=\m w_2^0, 
\EEQ
whose the energy
$$
E_0(t) = \frac{1}{2} \,  \left\|(w_1,\dot w_1,w_2) \right\|^2_{{\cal H}_0}, \, \, t \geq 0,
$$
where ${\cal H}_0:=H_{1,\half} \times H_1 \times H_2$, 
satisfies  the energy equality
\be
E_0(0) - E_0(t) = 
 \int_0^t
\left\|A_2^* w_2(s)\right\|_{H_2}^2\dd s,   \; \, t\geq 0.
\label{ESTENb}
\ee 
The second main result in this paper is to show that the exponential and polynomial decay of the energy $E$ of the abstract thermo-elastic system with Cattaneo law provides a polynomial decay of the energy $E_0$ of  the  classical thermo-elastic system \eqref{damped11}-\eqref{output11}. This is done by a spectral technic using a recent caracterization of polynomial stability of $C_0$-semigroups in Hilbert spaces due Borichev-Tomilov \cite{tomilov}.

Consider now the conservative adjoint  problem
\be
\ddot \phi_1(t) + A_1 \phi_1(t) + C \phi_2(t) = 0,
\label{eq3b}
\ee
\be
\label{eq3bb}
\dot \phi_2(t) + A_2 \phi_3 (t) - C^* \dot \phi_1(t) =0
\ee
\be
\label{eq3bb3}
\tau \, \dot \phi_3(t) - A_2^* \phi_2 (t) =0
\ee
\be 
\phi_1(0) = \phi_1^0, \, \dot \phi_1(0) = \phi_1^1, \, \phi_2(0) = \phi_2^0, \, \phi_3(0) = \phi^0_3,
\label{eq4b}
\ee
and the unbounded linear operators
\be
\label{opn}
{\cal A}_d : {\cal D}({\cal A}_d) \subset {\cal H} \rightarrow {\cal H}, \,
{\cal A}_d = \left(
\begin{array}{cccc}
0 &  I &  0  & 0\\
- A_1 & 0 & - C &  0\\
0 & C^* & 0 & - A_2\\
0 & 0 &  \frac{1}{\tau} \, A_2^* & - \frac{1}{\tau} I
\end{array}
\right),
\ee

\be
\label{opnx}
{\cal A}_c : {\cal D}({\cal A}_c) \subset {\cal H} \rightarrow {\cal H}, \,
{\cal A}_c = \left(
\begin{array}{cccc}
0 &  I &  0  & 0\\
- A_1 & 0 & - C &  0\\
0 & C^* & 0 & - A_2\\
0 & 0 &  \frac{1}{\tau} \, A_2^* & 0
\end{array}
\right),
\ee

\be
\label{opnv}
{\cal A} : {\cal D}({\cal A}) \subset {\cal H}_0 \rightarrow {\cal H}_0, \,
{\cal A} = \left(
\begin{array}{cccc}
0 &  I &  0  \\
- A_1 & 0 & - C \\
0 & C^* & - A
\end{array}
\right),
\ee
where
$$
{\cal D}({\cal A}_d) = {\cal D}({\cal A}_c) = H_{1,1} \times H_{1,\half} \times H_{2,\half} \times H_{1,\half},
$$
and
$$
{\cal D}({\cal A}) = H_{1,1} \times H_{1,\half}\times H_{2,1}.
$$
We transform the system \rfb{damped1}-\rfb{output} into a first-order system of evolution equation type.  For this,  let $W : = \left(w_1, \dot{w}_1,w_2,w_3 \right), \, W(0) = W^0 := \left(w_1^0,w_1^1,w_2^0,w_3^0 \right)$. Then, $W$ satisfies 
$$
\dot{W}(t) = {\cal A}_d W(t), \, t\geq 0, \quad W(0) = W^0.
$$
For the polynomial energy decay of the classical thermo-elastic system, we assume also the following assumption: \\
{\bf Assumption H}.
$i \, \rline \subset \rho({\cal A})$, where ${\cal A}$ is the operator defined by \rfb{opnv} and $\rho({\cal A})$ is the resolvent set of ${\cal A}$.

The main result of this paper is the following theorem.
\begin{theorem} \label{obscoup}
\begin{enumerate}
\item
The system described by \eqref{damped1}-\eqref{output} is
exponentially stable in ${\cal H}$ if and only if there exists $T,C > 0$ such that
$$
\int_{0}^{T} ||\phi_3(t)||^2_{H_1} \m\dd t \asymp \,
||(\phi^0_1,\phi^1_1,\phi_2^0,\phi_3^0)||^2_{{\cal H}}
$$
\be
\FORALL
(\phi^0_1,\phi^1_1,\phi^0_2,\phi_3^0) \in {\cal H}.
\label{CONDUNb}\ee
\item
If the system described by \eqref{damped1}-\eqref{output} is
exponentially stable in ${\cal H}$ then $(w_1,\dot{w}_1,w_2)$ solution of \eqref{damped11}-\eqref{output11} is polynomially stable for all initial data in $H_{1,1} \times H_{1,\half} \times H_{2,1}$, i.e., there exists a constant $C > 0$ such that for all $(w_1^0,w_1^1,w_2^0) \in {\cal D}({\cal A})$ we have 
\be
\label{polyest}
\left\|(w_1(t),\dot{w}_1(t),w_2(t)\right\|_{{\cal H}_0} \leq  \frac{C}{\sqrt{t}} \, \left\|(w_1^0,w_1^1,w_2^0)\right\|_{{\cal D}({\cal A})}, \, \forall \, t >0.
\ee
\item
If there exist $\alpha, T,C > 0$ such that
\be\label{CONDUNbx}
\int_{0}^{T} ||\phi_3(t)||^2_{H_1} \m\dd t \asymp \,
||(\phi^0_1,\phi^1_1,\phi_2^0,\phi_3^0)||^2_{{\cal H}_{-\alpha}}
\ee
for all $
(\phi^0_1,\phi^1_1,\phi^0_2,\phi_3^0) \in {\cal H}_{-\alpha} = H_{1,- \frac{\alpha - 1}{2}} \times H_{1,- \frac{\alpha}{2}} \times H_{2,- \frac{\alpha}{2}} \times H_{1, - \frac{\alpha}{2}}$
then, there exists a constant $C > 0$ such that for all $(w_1^0,w_1^1,w_2^0,w_3^0) \in {\cal D}({\cal A}_d)$ we have
\be
\label{polyestx}
E(t) \leq  \frac{C}{t^{\frac{1}{\alpha}}} \, \left\|(w_1^0,w_1^1,w_2^0,w_3^0)\right\|_{{\cal D}({\cal A}_d)}^2, \, \forall \, t >0.
\ee
\item
If the solution of the system described by \eqref{damped1}-\eqref{output} satisfies  \rfb{polyestx} then the  solution of \eqref{damped11}-\eqref{output11}  satisfies 
\be
\label{polyestxx}
E_0(t) \leq  \frac{C}{t^{\frac{1}{\alpha + 1}}  } \, \left\|(w_1^0,w_1^1,w_2^0)\right\|_{{\cal D}({\cal A})}^2, \, \forall \, t >0
\ee
for some    constant $C > 0$  and  all $(w_1^0,w_1^1,w_2^0) \in {\cal D}({\cal A})$.
\end{enumerate}
\end{theorem}
As a direct consequence we have the following corollary.
\begin{corollary}
\begin{enumerate}
\item
If the system  \eqref{damped1}-\eqref{output} satisfies \rfb{CONDUNb} for all initial data in ${\cal D}({\cal A}_d)$ then the system \eqref{damped11}-\eqref{output11} satisfies \rfb{polyest} for all initial data in ${\cal D({\cal A})}$. 
\item
If the system  \eqref{damped1}-\eqref{output} satisfies \rfb{CONDUNbx} for all initial data in ${\cal D}({\cal A}_d)$ then the system  \eqref{damped1}-\eqref{output} satisfies \rfb{polyestxx} for all initial data in ${\cal D}({\cal A})$. 
\end{enumerate}
\end{corollary}

The paper is organized as follows.  In Section 2, we show the
well-posedness of the evolution  system
\rfb{damped1}-\rfb{output}, by showing that the operator $({\cal
A}_d,{\cal D}({\cal A}_d))$ generates a contraction $C_0$-semigroup
in the space ${\cal H}$. In the third section we
give some results in the regularity for some infinite dimensional systems needed of the proof of the main result. Section \ref{proof} contains the proof of the main results. Some applications are given in Section \ref{appl}.

\section{Well-posedness}
\setcounter{equation}{0}
Let ${\cal H}:= H_{1,\half} \times H_1 \times H_2 \times H_1^\tau$ the Hilbert space endowed with the inner product
$$
\left<\left(
\begin{array}{ccc}u_1\\u_2\\u_3 \\ u_4\end{array}\right),\left(
\begin{array}{ccc}v_1\\v_2\\v_3 \\ v_4\end{array}\right)\right>_{{\cal H}}=
\left<A_1^\half u_1,A_1^\half v_1\right>_{H_{1}}+\left<u_2,v_2\right>_{H_1}+
 \left<u_3,v_3\right>_{H_2} + \tau \, \left<u_4,v_4\right>_{H_1}.
 $$
We have the following fundamental result.
\begin{theorem}\label{generation}  The operator
${\cal A}_d$, respectively ${\cal A}$, generates a strongly continuous contraction semigroup
$({\cal T}(t))_{t\geq0}$ on ${\cal H}$, respectively on ${\cal H}_0$.
\end{theorem}
\proof Take $\left(
\begin{array}
{ccc}u_1\\u_2\\v\\w
\end{array}
\right)\in {\cal D}({\cal A}_d)$. We have
\begin{align*}
\left<{\cal A}_d\left(
\begin{array}{ccc} u_1\\u_2\\v \\ w\end{array}\right),\left(
\begin{array}{ccc}u_1\\u_2\\v \\w \end{array}\right)\right>_{{\cal H}}&=\left<\left(
\begin{array}{l}
 u_2 \\
- A_1 u_1 -  C v  \\
  C^*  u_2-A_2w \\
\frac{1}{\tau}  \, A^*_2 v - \frac{1}{\tau} \, w
\end{array}
\right),\left(
\begin{array}{ccc}u_1\\u_2\\v\\w\end{array}\right)\right>_{{\cal H}}\\
&= - \, \|w\|^2_{H_1}.
\end{align*}
Thus ${\cal A}_d$ is dissipative.
%%%%%%%%%%%%%%%%%%%%%%%%%%%%%%%%%%%%%%%%%%%%%%%%%%%%%%%%%%%%%%%%%%%%%%%%%%%%%%%%%%
The density of ${\cal D}({\cal A}_d)$ is obvious.

Next, we are going to show that ${\cal A}_d$ is closed and 
\be
\label{adj}
{\cal D}({\cal A}^*_d) = {\cal D}({\cal A}_d), \, 
{\cal A}^*_d = \left(
\begin{array}{cccc}
0 &  I &  0  & 0\\
- A_1 & 0 & - C^* &  0\\
0 & C & 0 & - A_2^*\\
0 & 0 & \frac{1}{\tau} \, A_2 & - \frac{1}{\tau} I
\end{array}
\right).
\ee
Let $(W_n) \subset {\cal D}({\cal A}_d), \, W_n \rightarrow W \in {\cal H}, \, {\cal A}_d W_n \rightarrow Z \in 
{\cal H}$ as $n \rightarrow \infty.$
Then
$$
\left\langle {\cal A}_d W_n, \Phi \right\rangle_{{\cal H}} \rightarrow \left\langle Z,\Phi\right\rangle_{{\cal H}}.
$$
Choosing successively 
$
\Phi = (\Phi^1,0,0,0), \, \Phi^1 \in H_{1,1},\, \Phi = (0,0,\Phi^3,0), \, \Phi^3 \in H_{2,\half}, \, \Phi = (0,0,0,\Phi^4), \, \Phi^4 \in H_{1,\half},$ and $
\Phi = (0,\Phi^2,0,0), \, \Phi^2 \in  H_{1,\half}$, 
we obtain
$$
W^2 \in H_{1,\half},  W^2 = Z^1, \, W^4 \in H_{1,\half}, \, C^* W^2 - A_2 W^4 = Z^3, \, 
$$
$$
W^3 \in H_{2,\half}, \, A^*_2 W^3 - W^4 = \tau \, Z^3; \, W^1 \in H_{1,1}, \, 
$$
$$
- A_1 W^1 - C W^3 = Z^2,
$$
which yields that  $W \in {\cal D}({\cal A}_d)$ and ${\cal A}_d W = Z$.

$$ V \in {\cal D}({\cal A}^*_d) \Leftrightarrow \exists Z \in {\cal H} \, \forall \, \Phi \in {\cal D}({\cal A}_d); \,
\left\langle {\cal A}_d \Phi,Z\right\rangle_{\cal H} = \left\langle \Phi,Z\right\rangle_{\cal H}.
$$
Choosing $\Phi$ approprialtely as in above, the conclusion \rfb{adj} follows.
Finally, the Hille-Yosida theorem leads to the claim.

By the same way we can prove that ${\cal A}$ generates a $C_0$- semigroup of contractions on ${\cal H}_0$.\endproof

\section{Regularity of some coupled systems} \label{trans}
\setcounter{equation}{0}
We consider the initial and boundary value problems
\be  \ddot\phi_1(t) + A_1 \phi_1(t) + C \phi_2(t)= 0, \,
\dot\phi_2(t)  + A_2 \phi_3 - C^*\dot\phi_1(t)= 0, \, \tau \, \dot \phi_3(t) - A^*_2 \phi_2(t) =0 \label{eq3}\ee
\be \phi_1(0) = w_1^0, \, \dot\phi_1(0) = w_1^1, \, \phi_2(0) = w_2^0, \, \phi_3(0) = w_3^0,
\label{eq5}
\ee
and
\be
\label{eq11} \ddot{\phi}(t) + A_1 \phi(t) + C \psi(t) = 0, \,
\dot{\psi} + A_2 w(t) - C^* \dot \phi(t) = 0, \, \tau \, \dot w(t) - A^*_2 \psi(t) = g(t) 
\ee
\be
\label{eq18bis} \phi(0) = 0, \,\dot{\phi}(0)= 0,  \,
\psi(0) = 0, \, w(0) = 0. \ee
We have the following proposition.
\begin{proposition} \label{reg}
Let $g \in L^2(0,T;H_2)$. Then the system (\ref{eq11})-(\ref{eq18bis}) admits a unique
solution
\be
\label{eq19}
\left(\phi,\dot{\phi},\psi,w \right) \in C(0,T;H_{1,\half} \times H_1 \times H_2 \times H_1). 
\ee 
Moreover
$\ds w  \in L^2(0,T;H_1)$ and there
exists a constant $C >0$ such that
\be
\label{eq20}
 \left| \left| w \right| \right|_{L^2(0,T;H_2)} \leq C \, \left| \left|g \right|
\right|_{L^2(0,T;H_1)}, \, \forall \, g \in L^2(0,T;H_1).
\ee
\end{proposition}
For proving Proposition \ref{reg}, we should study the
conservative system  (without dissipation)
associated to problem (\ref{damped1})-(\ref{output}).
We have the following result.
\begin{lemma} \label{reg1}
For all $\left(w_1^0,w_1^1,w_2^0,w_3^0\right) \in H_{1,\half} \times H_1 \times H_2 \times H_1$ the system
\noindent (\ref{eq3})-(\ref{eq5}) admits a
unique solution $\left(\phi_1, \dot{\phi}_1,\phi_2, \phi_3 \right) \in C(0,T;H_{1,\half} \times H_1 \times H_2 \times H_1)$. Then $\ds \phi_3  \in L^2(0,T;H_1)$ and there exists a
constant  $C>0$ such that
\be
\left| \left|\ds \phi_3
\right| \right|_{L^2(0,T;H_1)} \leq C \, \left| \left|
(w_1^0,w_1^1,w_2^0,w_3^0) \right| \right|_{H_{1,\half} \times H_1\times H_2 \times H_1},
\label{eq29}
\ee
$$
\forall \,
(w_1^0,w_1^1,w_2^0,w_3^0) \in H_{1,\half} \times H_1 \times H_2 \times H_1.
$$
\end{lemma}
\begin{proof}
By the classical semigroup theory, see \cite{pazy}, we prove that for all $\left(w_1^0,w_1^1,w_2^0,w_3^0\right) \in H_{1,\half} \times H_1 \times H_2 \times H_1$ the system (\ref{eq3})-(\ref{eq5}) admits a
unique solution $\left(\phi_1, \dot{\phi}_1, \phi_2, \phi_3 \right) \in C(0,T;H_{1,\half} \times H_1\times H_2 \times H_1)$. We obtain that $\ds \phi_3 \in L^2(0,T;H_1)$ and that \eqref{eq29} holds.
\end{proof}

Now we can give the proof of  Proposition \ref{reg}.
\begin{proof} {\it of Proposition \ref{reg}.}

Let  the operator
\[
{\cal A}_c : {\cal D}({\cal A}_c) = H_{1,1} \times H_{1,\half} \times H_{2,\frac{1}{2}} \times H_{2,\frac{1}{2}} \subset {\cal H} \rightarrow {\cal H},
\]
defined by
\[
{\cal A}_c
\begin{pmatrix}
u_1 \cr
u_2 \cr
u_3 \cr u_4
\end{pmatrix}
 = \begin{pmatrix}
u_2 \cr - A_1u_1 - C u_3
\cr C^* u_2 \cr 
\frac{1}{\tau} \, A_2^*u_2 
\end{pmatrix}
, \, \forall \, (u_1,u_2,u_3,u_4) \in {\cal D}({\cal A}).
\]
${\cal A}_c$ is a skew-adjoint operator and
generates a group of isometries $(S(t))_{t\in \mathbb{R}}$ on
${\cal H}$.
Moreover we define the operator
\be
\label{DEFB0}
{\cal B}:H_{2} \rightarrow {\cal H},\
{\cal B} k= \begin{pmatrix}
0 \cr 0 \cr 0 \cr \frac{1}{\sqrt{\tau}} \, k
\end{pmatrix},\ \forall \, k \in H_{1}.
\ee
The problem \rfb{eq11}-\rfb{eq18bis} can be rewritten as a
Cauchy problem on ${\cal H}$
under the form
\be
\label{ECAV}
\begin{pmatrix}
\phi \cr  \dot{\phi} \cr \psi \cr w \end{pmatrix}^\prime (t)
={\cal A}_c \begin{pmatrix} \phi \cr \dot{\phi} \cr \psi \cr w \end{pmatrix}(t) - {\cal B} g(t), \, t > 0,
\ee
\be
\label{CIAV}
\phi(0) = 0, \,  \dot{\phi}(0) = 0, \, \psi(0)=0, \, w(0) = 0.
\ee
We can see that the operator
${\cal B}^*:{\cal H} \rightarrow H_1$ is given by
$$
{\cal B}^* \begin{pmatrix} u_1 \cr u_2 \cr v_1 \cr v_2 \end{pmatrix}= \frac{1}{\sqrt{\tau}} v_2,\FORALL
(u_1,u_2,v_1,v_2) \in {\cal H},
$$
which implies that
\be
\label{ACTADJ}
{\cal B}^*S^*(t) \begin{pmatrix} w_1^0\cr w_1^1 \cr w_2^0 \cr w_3^0  \end{pmatrix}={\cal B}^*\begin{pmatrix}\phi_1(t)\cr\dot \phi_1(t)\cr\phi_2(t) \cr \phi_3(t)  \end{pmatrix}=
\frac{1}{\sqrt{\tau}} \phi_3(t),
\FORALL (w_1^0,w_1^1,w_2^0,w_3^0)\in {\cal D}({\cal A}_c),
\ee
with $(\phi_1,\phi_2,\phi_3)$ is the solution of  \rfb{eq3}-\rfb{eq5}.
According to semigroup theory, see \cite{pazy}, we have that \rfb{eq11}-\rfb{eq18bis}
admits a unique solution
$$\left(\phi,\dot{\phi},\psi,w \right)(t) = \int_0^t S(t-s){\cal B}g(s) \, ds
\in C(0,T;{\cal H})
$$
which satisfies the regularity \rfb{eq20}.
  
\end{proof}

\section{Proof of the main result} \label{proof}
\setcounter{equation}{0}
Let $(w_1,\dot w_1,w_2,w_3) \in C(0,T; H_{1,\half} \times H_1 \times H_2 \times H_1)$
be the solution of \rfb{damped1}-\rfb{output} for a given initial data $(w_1^0,w_1^1,w_2^0,w_3^0)$.
Then $(w_1,\dot w_1,w_2,w_3)$ can be written as
\be
 (w_1,\dot w_1,w_2,w_3) =(\phi_1,\dot \phi_1,\phi_2,\phi_3)+(\phi,\dot \phi,\psi,w),
\label{SUMA}\ee
where $(\phi_1,\phi_2,\phi_3)$ satisfies \rfb{eq3}-\rfb{eq5}
and $(\phi,\psi,w)$ satisfies \rfb{eq11}-\rfb{eq18bis} with $g=  - w_3$.

The main ingredient of the proof of Theorem \ref{obscoup}
is the following result.
\begin{lemma}\label{echivalenta}
 Let
$(w_1^0, w_1^1,w_2^0,w_3^0) \in  H_{1,\half} \times H_1 \times H_2 \times H_1$.
Then the solution $(w_1,\dot w_1,w_2,w_3)$ of \rfb{damped1}-\rfb{output} and
the solution $(\phi_1,\phi_2,\phi_3)$ of
 \rfb{eq3}-\rfb{eq5}
satisfy
\be
C_1\int_0^T || \phi_3(t)||_{H_1}^2dt\le
\int_0^T ||w_3(t)||_{H_1}^2dt\le
4\int_0^T || \phi_3(t)||_{H_1}^2dt,
\label{INEGDUB}\ee
where $C_1>0$ is a  constant independent of $(w_1^0, w_1^1,w_2^0,w_3^0)$.
\end{lemma}
\begin{proof}
We prove \rfb{INEGDUB} for $(w_1,w_2,w_3)$ satisfying \rfb{damped1}-\rfb{output}
and $(\phi_1,\phi_2,\phi_3)$ solution of  \rfb{eq3}-\rfb{eq5}.
We know that $w_3 \in
L^2(0,T;H_1)$ and that \rfb{ESTEN} holds true.
Relation \rfb{SUMA} implies that
$$
\int_0^T ||\phi_3(t)||_{H_1}^2dt\le
2 \left\{ \int_{0}^{T}||w_3(t)||_{H_1}^2dt+
\int_0^T ||w(t)||_{H_1}^2dt\right\}.
$$
By applying now Proposition \ref{reg} with $g = - w_3
\in L^2(0,T;H_1)$ we obtain that
\be
\label{a}
\int_0^T ||w(t)||_{H_1}^2 \, dt\le C \,
\int_0^T || w_3(t)||_{H_1}^2 \, dt.
\ee
Then
the first inequality of \eqref{INEGDUB} holds true.

On the other hand, according to relation
\rfb{SUMA} we have that
\[
\phi_3 \in L^2(0,T;H_1),
\]
and
\be
\label{eq177bis}
\ddot \phi(t) + A_1 \phi(t) + C \psi(t) = 0, \, \dot \psi(t) + A_2 w(t) - C^* \dot \phi(t) =0, \, \dot{w}(t) - A_2^* \psi(t) + w(t) = - \phi_3(t).
\ee
We still denote by $\phi_3$ the extension by  $0, \, t \in \rline \setminus [0,T]$. We still also denote by $(\phi(t),\psi(t),w(t))$ the functions  $(1_{[0,T]}\phi(t),1_{[0,T]}\psi(t),1_{[0,T]} w(t))$. It is clear that these functions satisfy  the equation on the line $\rline$
\be
\label{psi}
\left\{
\begin{array}{ll}
\ddot \phi(t) + A_1 \phi(t) + C \psi(t) = 0, \,
\dot \psi(t) + A_2 w(t) - C^* \dot \phi(t) = 0, \\
\dot{w}(t) - A_2^* \psi(t) + w(t)= - \phi_3(t), \, t \in
\rline, \, 
\phi(0)= 0, \, \dot \phi (0) = 0, \, \psi(0) = 0, \, w(0) = 0.
\end{array}
\right.
\ee
Taking the Laplace transform we obtain
$$
\lambda^2 \widehat{\phi}(\lambda) + A_1 \widehat{\phi}(\lambda) + C \widehat{\psi}(\lambda) = 0, \,
 \lambda \, \widehat {\psi}(\lambda) + A_2 \widehat {w}(\lambda) -
\lambda \, C^* \widehat {\phi}(\lambda) = 0, 
$$
$$
\lambda \tau \, \widehat{w}(\lambda) - A^*_2 \widehat{\psi}(\lambda) + \widehat{w}(\lambda)= - \widehat{\phi}_3(\lambda), \quad \,
\forall \, \lambda = \gamma + i \eta, \, \gamma > 0.
$$
The equality above holds in $H_{1,-\half}, \, H_{2,-\frac{1}{2}}, H_{2,-\frac{1}{2}}$, respectively.
By applying $\lambda\bar{\widehat{\phi}} \in H_{1,\half}, \,
\bar{\widehat{\psi}} \in H_2, \, \bar{\widehat{w}} \in H_1$ respectively to first, second and to the third
equation on the  equalities above, we get by taking the real part,
$$
\gamma \, |\lambda|^2 \, ||\widehat {\phi}(\lambda)||^2_{H_1} +
\gamma \, ||A_1^\half \widehat {\phi}(\lambda)||_{H_1}^2 +  
 \gamma \,
||\widehat {\psi}(\lambda)||^2_{H_2} + (\gamma \tau + 1) \, \left\|\widehat{w}(\lambda)\right\|_{H_1}^2
 = 
 $$
 $$
 - \, \Re \, \left(
<\widehat{\phi}_3(\lambda), \bar{\widehat{w}}(\lambda)>_{H_1} \right).
$$
We get, 
\[
\int_{\rline_{\eta}} || \widehat {w}(\lambda)||^2_{H_1} \, d \eta
\leq \frac{1}{2} \int_{\rline_{\eta}} || \widehat {\phi}_3(\lambda)||^2_{H_1} \, d \eta +
\frac{1}{2}
\int_{\rline_{\eta}} ||\widehat {w}(\lambda)||^2_{H_1} \, d \eta.
\]
Parseval identity implies
\be
\label{ineqpsi}
\left\Vert w\right\Vert_{L^2(0,T;H_1)}^2\le
\left\Vert  \phi_3\right\Vert_{L^2(0,T;H_1)}^2,
\ee
and with relation \rfb{SUMA}, we have
\be
\left\Vert w_3\right\Vert_{L^2(0,T;H_1)}^2\le 4
\left\Vert  \phi_3\right\Vert_{L^2(0,T;H_1)}^2.
\label{ADOUAPARTE}\ee
This achieves the proof.
\end{proof}

We can now prove  Theorem \ref{obscoup}.
\\
{\it Proof of the first assertion }.
All finite energy solutions
of \rfb{damped1}-\rfb{output} satisfy the estimate
\be
E(t)\le Me^{-\omega t}E(0),\FORALL t\ge 0,
\label{EXPONENTIAL}\ee
where $M,\omega>0$ are constants independent of $(w_1^0,w^1_1,w^0_2,w_3^0)$,
if and only if there exist a time
$T>0$ and a constant $C>0$ (depending on $T$) such that
$$
E(0)-E(T)\ge CE(0),\FORALL (w_1^0,w^1_1,w^0_2,w_3^0) \in H_{1,\half} \times H_1 \times H_2 \times H_1.
$$
By \rfb{ESTEN} relation above is equivalent to the inequality
$$
\int_0^T ||w_3(s)||_{H_1}^2ds\ge C \, E(0),\FORALL
(w_1^0,w_1^1,w_2^0,w_3^0) \in H_{1,\half} \times H_1 \times H_2 \times H_1.
$$
\relax From Lemma \ref{echivalenta} it follows that the system \rfb{damped1}-\rfb{output}
is exponentially stable if and only if
\[
\int_0^T ||\phi_3(s)||_{H_1}^2ds\ge C \, E(0),\FORALL
(w_1^0,w^1_1,w^0_2,w_3^0) \in H_{1,\half} \times  H_{1} \times  H_2 \times H_1
\]
holds true. It follows that \rfb{damped1}-\rfb{output}
is exponentially stable if and only if
\rfb{CONDUNb} holds true. This ends up
the proof of the first assertion of Theorem \ref{obscoup}.

\noindent
{\it Proof of the third assertion }. \\

We have that for all 
$(\phi^0_1,\phi^1_{1},\phi_2^0,\phi^1_3) \in {\cal H}$ 
\be
\label{h3}
\int_{0}^{T} ||\phi_3(t)||^2_{H_1} \, dt \geq C \, 
||(\phi^0_1,\phi_1^1,\phi_2^0,\phi^1_3)||^2_{{\cal H}_{- \alpha}}.
\ee
Then, by Lemma \ref{echivalenta}
combined with \rfb{h3} and \rfb{ESTENb} imply the existence of a constant $K>0$
such that  
\[
||(w_1(T),w_1^\prime(T),w_2(T),w_3(T))||^2_{{\cal H}} \le
||(w_1^0,w_1^1,w_2^0,w_3^0)||^2_{{\cal H}}
-K \frac{||(w_1^0,w_1^1,w_2^0,w_3^0)||^{2 + 2 \alpha}_{{\cal H}_{- \alpha}}}{||(w_1^0,w_1^1,w_2^0,w_3^0)||^{2 \alpha}_{{\cal H}}},
\]
\be
\FORALL (w_1^0,w_1^1,w_2^0,w_3^0)\in {{\cal D}({\cal A}_{d})}.
\label{PREINT}\ee
Estimate (\ref{PREINT}) remains valid in successive intervals 
$[kT, (k+1)T]$ and since ${\cal A}_d$ generates a semigroup of contractions in 
${\cal D}({\cal A}_d)$ and the graph norm on ${\cal D}({\cal A}_d)$ is 
equivalent to $||. ||_{{\cal H}_1}$. We obtain the existence 
of a constant $C > 0$ such that for all $k \geq 0$ we have
$$
||(w_1((k+1)T),w_1^\prime((k+1)T),w_2((k+1)T),w_3((k+1)T))||^2_{{\cal H}}\le
$$
$$
||(w_1(kT),w_1^\prime(kT),w_2(kT), w_3(kT))||^2_{{\cal H}} -
$$
$$
-C \, \frac{||(w_1((k+1)T),w_1^\prime((k+1)T),w_2((k+1)T),w_3((k+1)T))||^{2 + 2 \alpha}_{{\cal H}}}
{||(w_1^0,w_1^1,w_2^0,w_3^0)||_{{\cal D}({\cal A}_d)}^{2 \alpha}},
$$
\be
\FORALL (w_1^0,w_1^1,w_2^0,w_3^0)\in {\cal D}({\cal A}_d).
\label{SUC}\ee
If we adopt the notation
\be {\cal H}_k= \frac{||(w_1(kT),w_1^\prime(kT),w_2(kT),w_3(kT))||^2_{{\cal H}}}
{||(w_1^0,w_1^1,w_2^0,w_3^0)||^2_{{\cal D}({\cal A}_d)}},\label{NOTEN}\ee
relation (\ref{SUC}) gives
\be  {\cal H}_{k+1}\le{\cal H}_k-
C{\cal H}^{1 + \alpha}_{k+1},\ \forall k\ge 0.\label{REC}\ee
By applying the following lemma.

\begin{lemma}\cite[Lemma 5.2]{sicon} \label{genrussell}
Let $({\cal E}_k)$ be a sequence of positive real numbers satisfying
\be {\cal E}_{k+1}\le {\cal E}_{k}- C {\cal E}_{k+1}^{2 + \delta},\
\forall k\ge 0,\label{INEGREC}\ee
where $C > 0$ and $\delta > - 1$ are constants. Then there exists 
a positive constant $M$ such that
\be {\cal E}_{k}\le\frac{M}{(k+1)^{\frac{1}{1 + \delta}}},\
\forall k\ge 0.\label{CONREC}\ee
\end{lemma}
and
using relation (\ref{REC}) we obtain
the existence of a constant $M>0$ such that
$$||(w_1(kT),w_1^\prime(kT),w_2(kT),w_3(kT))||^2_{{\cal H}}\le
\frac{M ||(w_1^0,w_1^1,w_2^0,w_3^0)||^2_{{\cal D}({\cal A}_d)}}{(k+1)^{\frac{1}{\alpha}}},\
\forall k\ge 0,$$
which obviously implies \rfb{polyestx}. 

\noindent
{\it Proof of the second assertion}. \\ 
The second assertion of Theorem \ref{obscoup} is equivalent to the following 

\be
\rho ({\cal A}_d)\supset \bigr\{i \beta \bigm|\beta \in \rline
\bigr\} \equiv i \rline, \label{1.8bv} \ee and \be \limsup_{|\beta|\to \infty }  \| (i\beta -{\cal A}_d)^{-1}\| <\infty
\label{1.9bv} 
\ee
implies that by a result of Borichev-Tomilov \cite{tomilov} that ${\cal A}$ satisfies the following two conditions: 
\be
\rho ({\cal A})\supset \bigr\{i \beta \bigm|\beta \in \rline
\bigr\} \equiv i \rline, \label{1.8b} \ee and \be \limsup_{|\beta|\to \infty } \, \frac{1}{\beta^{2}} \, \| (i\beta -{\cal A})^{-1}\| <\infty, 
\label{1.9b} 
\ee
where $\rho({\cal A})$, respectively $\rho({\cal A}_d),$ denotes the resolvent set of the operator
${\cal A}$, respectively of ${\cal A}_d$.

By assumption {\bf H} the conditions \rfb{1.8b}, \rfb{1.8bv} are
satisfied. Now for proving the above implication, suppose
that the condition \rfb{1.9b} is false. By the Banach-Steinhaus
Theorem, there exist a sequence of real
numbers $\beta_n \rightarrow\infty$ and a sequence of vectors
$Z_n= \begin{pmatrix} u_n \cr \varphi_n \cr \theta_n \end{pmatrix}\in {\cal D} ({\cal A})$ with
$\|Z_n\|_{\cH_0} = 1$ such that 
\be 
|| \beta_n^{2} \, (i \b_n I - {\cal
A})Z_n||_{\cH_0} \rightarrow 0\;\;\;\; \mbox{as}\;\;\;n\rightarrow
\infty, 
\label{1.12} 
\ee 
i.e., \be \beta_n^{2} \, \left( i \beta_n u_n - \varphi_n \right)
\rightarrow 0 \;\;\; \mbox{in}\;\; H_{1,\half}, \label{1.13}
 \ee 
 \be
 \beta_n^{2} \left( i \b_n \varphi_n + A_1 u_n + C \theta_n \right)  \rightarrow 0 \;\;\;
\mbox{in}\;\; H_{1},
\label{1.13bv} 
\ee
 \be 
\beta_n^{2} \, \left( i \b_n \theta_n + A \theta_n - C^* \varphi_n \right)  \rightarrow 0 \;\;\;
\mbox{in}\;\; H_{2}.  
\label{1.14} 
\ee 
We notice that we have
\be 
|| \beta_n^{2} (i \b_n I - {\cal A})Z_n||_{\cH_0} \ge |\Re \left(\langle \beta_n^{2} \, (i
\beta_n I - {\cal A})Z_n, Z_n\rangle_{\cH_0} \right)|. 
\label{1.15}
\ee
Then, by
\rfb{1.12} 
$$
\beta_n \, A_2^* \theta_n \rightarrow 0, \,  A_2^* \theta_n \rightarrow 0.
$$
Let $q_n = A_2^* \theta_n$,
\be
i \beta_n q_n + \frac {1}{\tau} q_n - A_2^* \theta_n \rightarrow 0, 
\label{1.16} \ee 
which implies that 
\be i \beta_n u_n - \varphi_n 
\rightarrow 0 \;\;\; \mbox{in}\;\; H_{1,\half}, \label{1.13b}
 \ee 
 \be
  i \b_n \varphi_n + A_1 u_n + C \theta_n   \rightarrow 0 \;\;\;
\mbox{in}\;\; H_{1},
\label{1.13bb} 
\ee
 \be 
 i \b_n \theta_n + A_2q_n - C^* \varphi_n  \rightarrow 0 \;\;\;
\mbox{in}\;\; H_{2}.  
\label{1.14b} 
\ee 
\be 
 i \b_n q_n  + \frac{1}{\tau} q_n -  A_2^*\theta_n  \rightarrow 0 \;\;\;
\mbox{in}\;\; H_{1}.  
\label{1.14bn} 
\ee 
i.e. 
$\tilde{Z}_n = \begin{pmatrix} u_n \cr \varphi_n \cr \theta_n \cr q_n \end{pmatrix}\in {\cal D} ({\cal A}_d)$ with
$\|\tilde{Z}_n\|_{\cH}$ bounded such that 
\be 
||  (i \b_n I - {\cal
A}_d)\tilde{Z}_n||_{\cH} \rightarrow 0\;\;\;\; \mbox{as}\;\;\;n\rightarrow
\infty, 
\label{1.12b} 
\ee 
which implies that \rfb{1.9bv} is false and ends the proof of the second assertion of Theorem \ref{obscoup}.

\noindent
{\it Proof of the fourth assertion of Theorem \ref{obscoup}}. \\ 
By the same way as above, we can prove the fourth assertion of Theorem \ref{obscoup}, i.e., 

the fourth assertion of Theorem \ref{obscoup} is equivalent to following: 
For $\alpha > 0,$
\be
\rho ({\cal A}_d)\supset \bigr\{i \beta \bigm|\beta \in \rline
\bigr\} \equiv i \rline, \label{1.8bvx} \ee and \be \limsup_{|\beta|\to \infty } \frac{1}{\beta^{2 \alpha}} \, \| (i\beta -{\cal A}_d)^{-1}\| <\infty, 
\label{1.9bvx} 
\ee
implies that by a result of Borichev-Tomilov \cite[Theorem 2.4 ]{tomilov} that ${\cal A}$ satisfies the following two conditions: 
\be
\rho ({\cal A})\supset \bigr\{i \beta \bigm|\beta \in \rline
\bigr\} \equiv i \rline, \label{1.8bx} \ee 
and \be \limsup_{|\beta|\to \infty }   \frac{1}{\beta^{2 \alpha + 2}} \, \| (i\beta -{\cal A})^{-1}\| <\infty.
\label{1.9bx} 
\ee

%\end{proof}

\section{Applications to stabilization for a thermo-elastic system}\label{appl}
\setcounter{equation}{0}

\subsection{ \bf First example} \label{sec122}

We consider the following initial and boundary problem
\be
\label{appl122}
 \left\{
\begin{array}{ll}
\ddot u_1 - \partial_x^2 u_1+ \partial_x u_2= 0, \, (0,+\infty) \times (0,1),\\
\dot u_2 - \partial_x u_3 +  \partial_x\dot u_1 =0, \, (0,+\infty) \times (0,1),\\
\tau \dot u_3 - \partial_x u_2 + u_3 = 0, \, (0,+\infty) \times (0,1),\\
u_1(t,0) = u_1(t,1) =0, \, (0,+\infty),\\
u_3(t,0) = u_3(t,1) = 0, \, (0,+\infty),\\
u_1(0,x) = u_1^0(x), \, \dot u_1(0,x) = u_1^1, \, u_2(0,x) = u_2^0, \, u_3(0,x) = u_3^0, \, x \in (0,1),
 \end{array}
\right.
\ee
where $0 < \tau$ and satisfies $\sqrt{\frac{\tau}{1 + \tau}} \notin \mathbb{Q}$. In this case, we have:
\[
H_1 = H_2 = L^2(0,1), \; H_{1,\half} = H^1_0(0,1),
\]
and
$$
A_1 = - \, \frac{d^2}{dx^2}, \, {\cal D}(A_1) = H^2(0,1) \cap H^1_0(0,1), \, A_2 = - \, \frac{d}{dx}, \, {\cal D}(A_2) = H^1(0,1),
$$
$$
A_2^* = \frac{d}{dx}, \, {\cal D}(A_2^*) = H^1_0(0,1),\, C = \frac{d}{dx} : H^1(0,1) \rightarrow L^2(0,1), 
$$
\be
\label{op122}
C^* = - \, \frac{d}{dx} : H^1_0(0,1) \rightarrow L^2(0,1).
\ee
Then, ${\cal A}_{d}$ is given by
$$
{\cal A}_d : {\cal D}({\cal A}_d) \rightarrow H_0^1(0,1) \times L^2(0,1) \times L^2(0,1) \times L^2(0,1), 
$$
$$
{\cal A}_d = \left(
\begin{array}{cccc}
0 & I &  0 & 0\\
\frac{d^2}{dx^2} & 0 & - \frac{d}{dx} & 0\\
0 & - \frac{d}{dx} & 0 & \frac{d}{dx} \\
0 & \frac{1}{\tau} \frac{d}{dx} & 0 & - \frac{1}{\tau} I
\end{array}
\right),
$$
where
\[
{\cal D}({\cal A}_d) = \left[  H^2(0,1) \cap H^1_0(0,1)\right] \times H^1_0(0,1) \times H^1(0,1) \times H^1_0(0,1).
\]
Stability results for \rfb{appl122}, for $\tau = 0$, are then a consequence of
Theorem \ref{obscoup}. 
%We have the following result

In this case the problem \rfb{eq3b}-\rfb{eq4b} becomes
\be
\label{HOM12}
\ddot \phi_1 - \partial^2_x \phi_1 +  \partial_x \phi_2 = 0, \, (0,1) \times
(0, + \infty),
\ee
\be
\label{hom112}
\dot \phi_2 - \partial_x \phi_3 + \partial_x \dot \phi_1 = 0, \, (0,1) \times
(0, + \infty),
\ee
\be
\label{cat2}
\tau \dot \phi_3 - \partial_x \phi_2 = 0, \, (0,1) \times
(0, + \infty),
\ee
\be
\label{HOM22}
\phi_i(0,t) = \phi_i(1,t) = 0, (0, + \infty),\, i=1,3,
\ee
\be
\label{HOM32}
\phi_i(x,0) = u^0_i(x), \, \dot \phi_1 (x,0) = u^1_1(x), \, (0,1), \, i=1,2,3.
\ee
The observability
inequality concerning the solutions
of \rfb{HOM12}-\rfb{HOM32} is given in the proposition below.
\begin{proposition}\label{obs22}
Let $T> 2$ be fixed. Then the following assertions hold true.

The solution $(\phi_1,\phi_2,\phi_3)$ of  \rfb{HOM12}-\rfb{HOM32} satisfies
$$
\int_0^T \int_0^1 \left|\phi_3(x,t)\right|^2 dx \, dt
\ge C \, \Vert (u^0_1,u^1_1,u^0_2,u_3^0)\Vert_{H^1_0(0,1) \times L^2(0,1) \times L^2(0,1) \times L^2(0,1)}^2,
$$
\be
 \forall (u^0_1, u^1_1,u^0_2,u_3^0) \in \dot{\cal H},
\label{OBSHOM12}
\ee
where $C >0$ is a constant and $$\dot{\cal H} = \left\{(u_1,u_2,u_3,u_4) \in H^1_0(0,1) \times L^2(0,1) \times L^2(0,1) \times L^2(0,1), \, \int_0^1 u_3(x) \, dx = 0 \right\} = 
$$
$$
\left\langle (0,0,1,0)^t\right\rangle^\bot.
$$
\end{proposition}
\begin{proof}
If we put
$$
\left(
\begin{array}{lll}
u^0_1
\\
u^1_1
\\
u^0_2
\\
u_3^0
\end{array}
\right) \in \dot{\cal H},\, i.e.,
\left(
\begin{array}{lll}
u^0_1
\\
u^1_1
\\
u^0_2
\\
u_3^0
\end{array}
\right)(x) = \sum_{n \in \mathbb{Z}^*} a_n  \varphi_n(x) 
$$
where  $(a_n)_{n \in \mathbb{Z}^*} \in l^2$, and 
$$
\varphi_n (x) = \left(
\begin{array}{lll}
 \sin (n \pi x)
\\
\lambda_n \, \sin(n \pi x)\\
- \, n\pi \, \left( \frac{1}{\tau \, \lambda_n} (\lambda_n - \frac{n\pi}{\lambda_n}) + 1 \right) \, \cos (n \pi x) \\
\frac{1}{\tau} \left( \lambda_n - \frac{n \pi}{\lambda_n} \right) \, \sin(n \pi x)
\end{array}
\right), \, n \in \mathbb{Z}^*,
$$
with
$$
(\lambda_n)_{n \in \mathbb{Z}^*} = \left\{i n \pi \sqrt{\frac{1 + \tau}{\tau}}, \, n \in \mathbb{Z}^* \right\} \cup \left\{i n \pi, \, n \in \zline^* \right\}.
$$
Then, we clearly have
\be
\phi_3(x,t)=
\sum_{n \in \mathbb{Z}^*} a_n \, \frac{1}{\tau} \left( \lambda_n - \frac{n \pi}{\lambda_n} \right) \, e^{\lambda_n t} \, \sin(n \pi x).
\label{1FORMEX0}
\ee
From Ingham's inequality (see Ingham \cite{ingham}) we obtain,
for all $T> 2$,
 the existence of a constant $C_T>0$ such that the solution $(\phi_1,\phi_2,\phi_3)$ of
\rfb{HOM12}-\rfb{HOM32} satisfies
\be
\int_0^T \int_0^1 \left|\phi_3(x,t)\right|^2 \, dx \, dt\ge
C_T \, \sum_{n \in \, \mathbb{Z}^*} | \lambda_n \, a_n|^2,
\label{1INGDIR}
\ee
which is exactly \rfb{OBSHOM12}.
\end{proof}

Now, as an immediate consequence of
Theorem \ref{obscoup} we have the following stability result for $(u_1,\dot{u}_1,u_2)$ solution of \rfb{appl122} with $\tau=0$.

\begin{proposition}
There exists a constant $C > 0$ such that for all 
$ (u^0_1,u_1^1,u_2^0) \in {\cal D}(\dot{{\cal A}}) = {\cal D}({\cal A}) \cap \dot{{\cal H}},$
$$
\left\| (u_1,\dot{u}_1,u_2) \right\|_{{\cal H}} \leq \frac{C}{\sqrt{t}} \left\| (u^0_1,u^1_1,u_2^0)\right\|_{{\cal D}({\cal A})}, \, \forall \, t > 0.
$$
\end{proposition}

\begin{remark}
We can obtain the same result, as above, by application of an exponential stability result obtained 
by Racke for \rfb{appl122} in \cite[Theorem 2.1]{rackeb} and Theorem \ref{obscoup}. 
\end{remark}

\subsection{\bf Second example} \label{sec123}
Let $\Omega$ be a bounded smooth domain of $\rline^2$. 
We consider the following initial and boundary problem:
\be
\label{appl122x}
 \left\{
\begin{array}{ll}
\ddot u - \mu \, \Delta u -  (\lambda + \mu) \nabla div u + \nabla \theta = 0, \, (0,+\infty) \times \Omega,\\
\dot \theta + div q + div \dot u =0, \, (0,+\infty) \times (0,1),\\
\tau \dot q + \nabla \theta  + q = 0, \, (0,+\infty) \times (0,1),\\
u =0, \theta =0,  \, \Omega \times (0,+\infty),\\
u(0,x) = u^0(x), \, \dot u(0,x) = u^1(x), \, \theta(0,x) = \theta^0, \, q(0,x) = q^0(x), \, x \in \Omega,
 \end{array}
\right.
\ee
The parameters $\tau, \mu, \lambda$ are positive constants which satisfy $\lambda + 2 \mu > 0$.

In this case, we have:
\[
H_1 = (L^2(\Omega))^2, H_2 = L^2(\Omega), \; H_{1,\half} = (H^1_0(\Omega))^2,
\]
and
$$
A_1 = - \mu \, \Delta - (\mu + \lambda) \nabla div, \, {\cal D}(A_1) = (H^2(\Omega) \cap H^1_0(\Omega))^2, \, A_2 =  div, \, {\cal D}(A_2) = H^1(\Omega),
$$
$$
A_2^* = - \nabla, \, {\cal D}(A_2^*) = (H^1_0(\Omega))^2,\, C = \nabla : H^1(\Omega) \rightarrow (L^2(\Omega))^2, 
$$
\be
\label{op122x}
C^* = - \, div : (H^1_0(\Omega))^2 \rightarrow L^2(\Omega).
\ee
Then, ${\cal A}_{d}$ is given by
$$
{\cal A}_d : {\cal D}({\cal A}_d) \rightarrow (H_0^1(\Omega) \times L^2(\Omega))^2 \times L^2(\Omega) \times (L^2(\Omega))^2, 
$$
where
\[
{\cal D}({\cal A}_d) = \left[  H^2(\Omega) \cap H^1_0(\Omega)\right]^2 \times (H^1_0(\Omega))^2 \times H^1_0(\Omega) \times (H^1(\Omega))^2.
\]
Stability result for \rfb{appl122x}, with $\tau =0$, are then an immediate consequence of
Theorem \ref{obscoup} and of \cite[Theorem 3.1]{rr}. We have the following result

\begin{proposition}
Let $\Omega$ be a a radially symmetric and let the initial data $(u^0,u^1,\theta^0, q^0)$ be radially symmetric \footnote{see \cite[page 327]{racke} for definitions.}. Then, there exists a constant $C > 0$ such that for all 
$ (u^0,u^1,\theta^0) \in [H^2(\Omega) \cap H^1_0(\Omega)] \times H^1_0(\Omega) \times [H^2(\Omega) \cap H^1(\Omega)],$
$$
\left\| (u,\dot{u},\theta) \right\|_{{\cal H}} \leq \frac{C}{\sqrt{t}} \left\| (u^0,u^1,\theta^0)\right\|_{[H^2(\Omega) \cap H^1_0(\Omega)]^2 \times H^1_0(\Omega) \times [H^2(\Omega) \cap H^1_0(\Omega)]}, \, \forall \, t > 0.
$$
\end{proposition}
\begin{remark}
We remark that we obtain the same stability result as Lebeau-Zuazua in \cite{lebeauzuazua2}.
\end{remark}


\begin{thebibliography}{99}

\bibitem{Benhassi1} {\sc E. M.~Ait Ben hassi, K.~Ammari,  S.~Boulite and L.~Maniar,} Feedback stabilization of a class of evolution equations with delay, {\em J. Evol. Equ.,}  {\bf9} (2009), 103--121.

\bibitem{Benhassi2} {\sc E.~M.~Ait Ben hassi,  K.~Ammari , S.~Boulite and L.~Maniar,} Stabilization of coupled second order systems with delay, {\em Semigroup Forum.,} {\bf86} (2013), 362--382.

\bibitem{ammari} {\sc K.~Ammari and M.~Tucsnak}, Stabilization of second order evolution equations by a class of unbounded feedbacks, {\em ESAIM COCV.,} {\bf 6} (2001), 361--386.

\bibitem{sicon} {\sc K.~Ammari and M.~Tucsnak,} Stabilization of Bernoulli-Euler beams by means of a pointwise feedbck force, {\em SIAM J CONTROL OPTIM.,} 
{\bf 39}, (2000), 1160--1181.

\bibitem{tomilov} {\sc A.~Borichev and Y. ~Tomilov}, Optimal polynomial decay of functions and operator semigroups, {\em Math. Ann.,} {\bf 347} (2010), 455--478.

\bibitem{haraux} {\sc A.~Haraux,} Une remarque sur la stabilisation de certains syst\`emes du deuxi\`eme ordre en temps, {\em Portugal. Math.,} {\bf 46} (1989), 245--258.

\bibitem{henry} {\sc D.~Henry, O.~Lopes and A. ~Perissinotto,} On the essential spectrum of a semigroup of thermoelasticity, {\em Nonlinear Anal. T.M.A.,} {\bf 21} (1993), 65--75.

\bibitem{racke} {\sc S.~Jiang and R.~Racke},  {\em Evolution equations in thermoelasticity.} Chapman and Hall/CRC Monographs and Surveys in Pure and Applied Mathematics, 112. Chapman and Hall/CRC, Boca Raton, FL, 2000.

\bibitem{ingham}  {\sc V.~Komornik and P.~Loreti}, {\em Fourier series in control theory,}
Springer Monographs in Mathematics, New York, 2005.

\bibitem{lebeauzuazua2} {\sc G.~Lebeau and E.~Zuazua}, Decay rates for the three-dimensional linear system of thermoelasticity, {\em Arch. Ration. Mech. Anal.,} {\bf 148} (1999), 179--231.

\bibitem{pazy} {\sc A.~Pazy}, {\em Semigroups of linear operators and application to partial differential equations,} Springer Verlag, New york, 1983.

\bibitem{rr} {\sc R.~Racke},  Asymptotic behavior of solutions in linear $2-$ or $3-D$ thermoelasticity with second sound, {\em Quart. Appl. Math.,} {\bf 61} (2003), 315--328. 

\bibitem{rackeb} {\sc R.~Racke,} Thermoelasticity with second sound-exponential stability in linear and non-linear 1-d, {\em Math. Meth. Appl. Sci.,} 
{\bf 25} (2002), 409--441.

\end{thebibliography}
\end{document}